\documentclass[11pt]{amsart}

\usepackage{amssymb}
\usepackage{enumerate}
\usepackage{hyperref}
\usepackage{textcomp}
\usepackage{filecontents}

\newcommand{\ideal}{\trianglelefteq}
\newcommand{\iso}{\simeq}
\newcommand{\ess}{\leq^{ess}}
\newcommand{\ness}{\nleq^{ess}}
\newcommand{\dsum}{\leq^{\oplus}}
\newcommand{\tor}{\mathcal{T}}
\newcommand{\torf}{\mathcal{F}}
\newcommand{\stor}{\mathcal{T}_\beta}
\newcommand{\storfree}{\mathcal{F}_\beta}
\newcommand{\sB}{\mathfrak{sB}}
\newcommand{\sR}{\mathfrak{sR}}
\newcommand{\sbcore}{\mathfrak{C}_{\mathfrak{sB}}}
\newcommand{\srcore}{\mathfrak{C}_{\mathfrak{sR}}}
\newcommand{\B}{\beta}

\newcommand{\Q}{\mathbb{Q}}

\newcommand{\Z}{\mathbb{Z}}
\newcommand{\modr}{\mathcal{M}od_R}
\newcommand{\rmod}{_R\mathcal{M}od}
\newcommand{\semicentral}{\mathcal{S}_l (R)}
\newcommand{\rsemicentral}{\mathcal{S}_r (R)}

\newtheorem{thm}{Theorem}[section]
\newtheorem{lem}[thm]{Lemma}
\newtheorem{defn}[thm]{Definition}
\newtheorem{ex}[thm]{Example}

\newtheorem{prop}[thm]{Proposition}
\newtheorem{cor}[thm]{Corollary}

\makeatletter
\newcommand{\addresseshere}{%
  \enddoc@text\let\enddoc@text\relax
}
\makeatother

\begin{document}

\title{s.Baer and s.Rickart Modules}

\author[Gary F. Birkenmeier]{Gary F. Birkenmeier\textsuperscript{$\dagger$}}
\address{\textsuperscript{$\dagger$}Department of Mathematics, University of Louisiana at Lafayette}
\email{gfb1127@louisiana.edu}

\author[Richard L. LeBlanc]{Richard L. LeBlanc\textsuperscript{$\bigstar$}}
\address{\textsuperscript{$\bigstar$}Department of Mathematics, Saint Mary's Hall}
\email[Corresponding author]{rleblanc@smhall.org}



\maketitle

\addresseshere

\begin{abstract}
In this paper, we study module theoretic definitions of the Baer and related ring concepts. We say a module is s.Baer if the right annihilator of a nonempty subset of the module is generated by an idempotent in the ring. We show that s.Baer modules satisfy a number of closure properties. Under certain conditions, a torsion theory is established for the s.Baer modules, and we provide examples of s.Baer torsion modules and modules with a nonzero s.Baer radical. The other principal interest of this paper is to provide explicit connections between s.Baer modules and projective modules. Among other results, we show that every s.Baer module is an essential extension of a projective module. Additionally, we prove, with limited and natural assumptions, that in a generalized triangular matrix ring every s.Baer submodule of the ring is projective. As an application, we show that every prime ring with a minimal right ideal has the strong summand intersection property. Numerous examples are provided to illustrate, motivate, and delimit the theory.
\end{abstract}


\section*{Introduction}
A \textit{Baer} ring is a ring in which the right annihilator of an arbitrary nonempty subset is generated by an idempotent. A more general notion of a Baer ring is that of a \textit{right Rickart} ring where the right annihilator of an arbitrary element is generated by an idempotent. A ring is right Rickart if and only if every principal right ideal is projective. Hence these rings are often referred to as \textit{right p.p.} rings. Baer and Rickart rings have a long history dating back to the 1940s with roots in functional analysis. For more on these topics see \cite{Be}, \cite{BPR}, \cite{E}, \cite{G1}, \cite{H}, \cite{K1}, and \cite{Rick}.

In 1972, Evans defined the p.p. condition in the module setting \cite{Ev}. He called a module a $c.p.$ module if every cyclic submodule is projective. In 2004, Lee and Zhou \cite{LZ} also looked at this same condition for modules but under a different name, p.p. modules. More recently in 2011, the Rickart condition was studied further in the module theoretic setting by Liu and Chen \cite{LC}. Aside from looking at p.p. modules, Lee and Zhou introduced a notion of Baer modules. For them, a module $M_R$ is called $Baer$ if, for any nonempty subset $S$ in $M$, $r_R(S) = eR$ where $e=e^2 \in R$ (also see \cite{K}). Also, in 2004, Rizvi and Roman studied the Baer ring concept in the module theoretic setting with respect to the endomorphism ring in contrast to \cite{Ev} and \cite{LZ} . Consider a right $R$-module $M$ and let $S = End_R(M)$. For \cite{RR}, $M$ is a $Baer$ module if the right annihilator in $M$ of any left ideal of $S$ is generated by an idempotent of $S$.

The notions of a Baer and Rickart module that we shall consider in this paper are exactly the definitions used by Evans, Lee and Zhou, and Liu and Chen. Thus, a module $M_R$ is called \textit{s.Rickart} if, for any $m \in M_R$, $r_R(m) = eR$ for some $e=e^2 \in R$. A module $M_R$ is called \textit{s.Baer} if, for any nonempty subset $S$ of $M$, $r_R(S) = eR$ for $e = e^2 \in R$. To contrast, we denote the Baer module concept of \cite{RR} by \textit{e.Baer}. Note that when $M_R = R_R$ all the aforementioned notions of a Baer module coincide.

In Section \ref{preliminaries}, we investigate a number of closure properties for s.Baer modules: submodules, direct sums, direct products, and module extensions. When $R$ has the SSIP or is orthogonally finite, the classes of s.Baer and s.Rickart modules coincide and are closed under direct products. We determine conditions on (s.Baer) s.Rickart modules which ensure that $R$ has the (S)SIP. For a simple module $M$, $M$ is nonsingular $\iff$ projective $\iff$ s.Rickart $\iff$ s.Baer. Then we characterize the primitive rings which have a faithful simple s.Baer module. Surprisingly, we prove that a right primitive ring with nonzero socle has the SSIP. If $M$ is s.Baer, we show when $Hom(M,-)$ and $Hom(-,M)$ are s.Baer. A ring $R$ is semisimple Artinian if and only if every $R$-module is s.Baer. Finally, we discuss conditions on $R$ such that all nonsingular modules are s.Baer.

In Section \ref{projectivity_section}, we begin exploring connections with projectivity. In particular, every s.Rickart module is an essential extension of a projective module (Theorem \ref{essprojective}). For the main result of the section, we determine a class of generalized triangular matrix rings which have a largest s.Rickart submodule. We also characterize when the s.Rickart submodules of a 2-by-2 upper triangular matrix ring over a domain are projective. 

Developing a torsion theory for the class of s.Baer modules forms the basis of Section \ref{torsion_section}. The class of s.Baer modules is a torsion-free class if and only if it is closed under direct products.  We show, in general, the s.Baer torsion theory is stable but not hereditary and we provide instances when it is hereditary. Our results culminate in the following statement: If $R$ is a semiprime ring which has the SSIP or is orthogonally finite, then every projective module splits into a direct sum of a s.Baer torsion module and a s.Baer torsion-free module.

This research is a part of Richard L. LeBlanc's Ph.D. thesis written under the supervision of Professor Gary F. Birkenmeier. Throughout this paper, all rings are associative with unity and $R$ denotes such a ring. All modules are unital right $R$-modules unless otherwise indicated. $\modr$ ($\rmod$) denotes the category of all right (left) $R$-modules, $\sB \ (\sR)$ the class of all s.Baer (s.Rickart) $R$-modules, and $M_R$ ($_RM$) a right (left) $R$-module. Module homomorphisms are written on the opposite side of the scalar. For $N \subseteq M$, $N_R \leq M_R$, $N_R \leq^{\oplus} M_R$, and $N_R \leq^{ess} M_R$  denote a subset, submodule, direct summand, and an essential submodule of $M$, respectively. We symbolize fully invariant submodules (ideals of $R$) by $N_R \trianglelefteq M_R$ ($I \trianglelefteq R$). The right annihilator in the ring $R$ is written $r_R( - )$ and the left annihilator in the module $M$ is written $l_M(-)$. $T_2 (R)$ is the ring of upper triangular matrices over R and $\langle - \rangle $ is the subring of $R$ generated by $-$. An idempotent $e$ is right (left) semicentral if, for any $x \in R$, $ex = exe$ ($xe = exe$). The set of all right (left) semicentral idempotents is $\mathcal{S}_r(R)$ ($\mathcal{S}_l(R)$) and $\mathcal{B} (R) = \semicentral \cap \rsemicentral$ is the set of central idempotents of $R$. $Z(M)$ and $Z_2(M)$ signify the singular submodule and the second singular submodule of $M$. The injective hull is $E(M)$. A module $M_R$ has the $(S)SIP$ if and only if a (arbitrary) finite intersection of direct summands is again a direct summand. The following result from \cite[Proposition 1.2]{S} will be used implicitly throughout this paper: for $X,Y \in \modr$ if $X \ess Y$, then for all $y \in Y$, $y^{-1}X = \{ r \in R \mid yr \in X \} \ess R_R$. Lastly, undefined notation or terminology can be found in \cite{BM}, \cite{BPR}, \cite{K}, and \cite{L}.

\section{Preliminary Results and Examples} \label{preliminaries}

To distinguish the various notions of the Baer module concept we introduce the following terminology. 

\begin{defn} \label{1stdefn}
(See \cite{Ev},\cite{K},\cite{LZ},\cite{LRR},\cite{RR}) Let $M \in \mathcal{M}od_R$ and let $S = End_R(M)$. 
\begin{enumerate}[(i)]
\item A module $M$ is s.Baer (scalar Baer) if, for any $\varnothing \neq N \subseteq M$, we have that $r_R(N) = eR$ for some $e=e^2 \in R$.

\item A module $M$ is e.Baer (endomorphism Baer) if, for any $N_R \leq M_R$, $l_S(N) = Se$ for some $e=e^2 \in S$.

\item A module $M$ is s.Rickart if, for any $n \in M$, we have that $r_R(n) = eR$ for some $e=e^2 \in R$.

\item A module $M$ is e.Rickart if, for all $\varphi \in S$, $r_M(\varphi) = eM$ for some $e = e^2 \in S$. Note that $r_M(\varphi) = r_M(S\varphi) = ker(\varphi)$.
\end{enumerate}
\end{defn}

By a Zorn's lemma argument every module contains a submodule maximal with respect to being s.Rickart. Clearly, s.Baer implies s.Rickart, and a s.Rickart module is nonsingular.

\begin{ex} \label{prelim_example}
The following examples distinguish s.Baer and e.Baer modules.
\end{ex}
\begin{enumerate}[(i)]

\item Let $R$ be a commutative domain. As we will see in Corollary \ref{baercor}, every submodule of a free module is s.Baer. However, if $R$ is not Pr\"{u}fer, then a free module of finite rank $> 1$ is not e.Baer \cite[Theorem 3.9]{RR1}.	

\item A direct computation will show that $M_\Z = \Q \oplus \Z_p$ is e.Baer since $End_{\Z}(M) \iso \left( \begin{smallmatrix} \Q & 0 \\ 0 & \Z_p \\ \end{smallmatrix} \right)$. We can see $M_\Z$ is not s.Baer since $r_\Z ( (0,1) ) = p \Z$ for any prime $p \in \Z$.

\item Let $R$ be $T_2 ( \Z )$.
\begin{enumerate}
\item If $e = \left( \begin{smallmatrix} 1 & 0 \\ 0 & 0 \\ \end{smallmatrix} \right)$, then $M = eR = \left( \begin{smallmatrix} \Z & \Z \\ 0 & 0 \\ \end{smallmatrix} \right)$ is e.Baer since $End_R(M) \iso \Z$ but not s.Baer since $r_R( \left( \begin{smallmatrix} 2 & 3 \\ 0 & 0 \\ \end{smallmatrix} \right))$ is not generated by an idempotent.
\item Let $K = \left( \begin{smallmatrix} 0 & \Z \\ 0 & \Z \\ \end{smallmatrix} \right)$. Then $End_R(K) \iso \left( \begin{smallmatrix} \Z & \Z \\ \Z & \Z \\ \end{smallmatrix} \right)$ is a Baer ring \cite[p. 17]{K1} and $K$ is retractable. From \cite[Proposition 4.6]{RR}, $K$ is e.Baer. Moreover, $K$ is s.Baer by Theorem \ref{gen_triangular_matrix_ring_thm} (i).
\end{enumerate}
\end{enumerate}


\begin{lem} \label{hered}
Let $M \in \mathcal{M}od_R$. A submodule of a s.Baer (s.Rickart) module $M_R$ is s.Baer (s.Rickart).
\end{lem}

\begin{proof} This result follows directly from the definitions. \end{proof}


\begin{lem} \label{rightlemma} \cite[Lemma 2.2]{LC} Let $M \in \modr$. Then $aR \cap r_R(X) = a\cdot r_R(Xa)$ for any $a \in R$ and $\varnothing \neq X \subseteq M$.
\end{lem}

\begin{lem} \label{rightsummand} Let $M \in \modr$ be s.Baer. Then $eR \cap r_R(X)$ is a direct summand of $R_R$ for any $e = e^2$ and $\varnothing \neq X \subseteq M$. 
\end{lem}

\begin{proof}
By Lemma \ref{rightlemma}, $eR \cap r_R(X) = e \cdot r_R(Xe)$. Since $M$ is s.Baer, $r_R(Xe) = fR$ for some $f=f^2 \in R$. Now, $e \cdot r_R(Xe) \subseteq r_R(Xe)$. Then $efR = e \cdot r_R(Xe) \subseteq r_R(Xe) = fR$. Thus $ef = fef$ and $(ef)^2 = efef = e(ef) = ef$. Hence, $eR \cap r_R(X) = efR$ is a direct summand of $R_R$.
\end{proof}


\begin{lem} \label{sRickart_class_extensions}
The class of s.Baer (s.Rickart) modules is closed under extensions.
\end{lem}
\begin{proof}
Let $K_R \leq M_R$ and suppose $K$ and $M/K$ are s.Baer. Now, let $\varnothing \neq S \subseteq  M$ and $\overline{S} = \{ s + K \mid s \in S \}$. Consider $r_R(\overline{S}) = \{ a \in R \ | \ Sa \subseteq K \} = eR$ for some $e=e^2 \in R$. Observe $r_R(S) \subseteq r_R(\overline{S})$, and $Se \subseteq K$. For $\alpha \in r_R(S)\subseteq eR$, $\alpha \in r_R(S) \cap r_R(Se) \subseteq eR \cap r_R(Se)$. By Lemma \ref{rightsummand}, $r_R(S) \subseteq gR$. Observe $Sg = Seg = 0$ so $g \in r_R(m)$. Therefore $r_R(S) = gR$. Hence $M$ is s.Baer. For the s.Rickart case, replace $S$ with $m \in M$ and apply \cite[Lemma 2.3]{LC}.
\end{proof}


\begin{prop} Let $M \in \modr$ and let $\{ M_\alpha \}_{\alpha \in A}$ be a family of modules such that $M_\alpha \iso M$ for all $\alpha \in A$. Then $M$ is s.Baer if and only if $\bigoplus_A M_\alpha$ is s.Baer if and only if $\prod_A M_\alpha$ is s.Baer.
\label{direct_prod_same} \end{prop}
\begin{proof}
By Lemma \ref{hered}, $\prod_A M_\alpha$ is s.Baer implies $\bigoplus_A M_\alpha$ is s.Baer and so $M_\alpha \iso M$ is s.Baer. Now suppose $M$ is s.Baer and consider $\prod_A M_\alpha$. Since $M_\alpha \iso M$, $\prod_A M_\alpha \iso \prod_A M$. Let $\varnothing \neq S \subseteq \prod_A M$. Then $S \subseteq \{ f: A \rightarrow M \mid f(\alpha) \in M \}$ and define $X$ to be $ \bigcup_{f \in X} \{ f(\alpha) \mid \alpha \in A \}$. Hence $r_R(S) = r_R(X) = eR$ for some $e=e^2 \in R$, since $M$ is s.Baer. Thus $\prod_A M_\alpha$ is s.Baer.
\end{proof}


\begin{prop} \label{finite_dsum} Let $\{ M_\alpha \}_{\alpha \in A}$ be a family of modules where $A$ is a finite index set. Then for every $\alpha \in A$, $M_\alpha$ is s.Baer if and only if $\bigoplus_A M_\alpha$ is s.Baer.
\end{prop}

\begin{proof}
By Lemma \ref{hered}, if $\bigoplus_A M_\alpha$ is s.Baer then so are the $M_\alpha$. Now suppose $\{ M_\alpha \}_{\alpha \in A}$ is a family of s.Baer modules. It suffices to show for $A = \{ 1, 2 \}$ that $M_1 \oplus M_2$ is s.Baer. Let $S \subseteq \{ f:A \rightarrow M_1 \cup M_2 \mid f(\alpha) \in M_\alpha \}$ and let $S_{f(1)} = \{ f(1) \mid f \in S \}$. Then $r_R (S_{f(1)}) = eR$ for $e = e^2 \in R$ since $M_1$ is s.Baer. Now $r_R(S) = eR \cap r_R(S_{f(2)}) \leq^{\oplus} R_R$ by Lemma \ref{rightsummand}. Therefore $M_1 \oplus M_2$ is s.Baer.
\end{proof}

In \cite[Proposition 2.22]{RR}, it is shown that a Baer ring has the SSIP. The following example shows there are SIP rings $R$ which are non-SSIP that have a nonzero s.Baer module.

\begin{ex} \label{counterexample} An infinite direct sum of s.Baer modules is not necessarily s.Baer.
\end{ex}

\noindent Let $I=\Z^+$, $M= \bigoplus_I F_i$, and $R = \langle \bigoplus_I F_i ,1 \rangle \subseteq \prod_I F_i$, where $F$ is a field and $F_i = F$ for all $i \in I$. Note that $R$ is a Rickart ring (hence, $R_R$ has the SIP) that is not Baer and $\prod_I F_i$ is a Baer ring. For any $0 \neq m \in M$, we claim that $mR$ is s.Baer. Without loss of generality, $mR = E_kR$ for some $k \in I$, where $E_n = \Sigma_{k=1}^{n} e_{i_k}$ and $e_{i_k}(j) \in R$ is $\delta_{i_kj}$ for $i_k,j \geq 1$. Let $\varnothing \neq S \subseteq mR$ and consider $SR$. Again without loss of generality, $SR = E_sR$ where $E_s$ is an idempotent and $s \leq k$. Then $r_R(S) = (1-E_s)R \Rightarrow mR$ is s.Baer. If $R$ has the SSIP, then $R$ is a Baer ring because $R$ is right Rickart (see \cite[Proposition 2.22]{RR}). But since $R$ is not Baer, $R$ does not satisfy the SSIP. Additionally note that $M$ is a faithful s.Rickart $R$-module but it is not s.Baer \cite[Theorem 2.7]{LC}. Next, let $N = \bigoplus_{i=1}^{\infty} E_iR$. Then $r_R(N) = \bigcap_{i=1}^{\infty} (1-E_i)R \nleq^{\oplus} R$. Thus $N_R$ is not s.Baer.

However our next result guarantees that if there is a nonzero s.Baer module $M \in \modr$, then there exists a nonzero factor  of $R$ that has the SSIP.


\begin{thm} Let $M \in \modr$ be s.Baer. Then:
\begin{enumerate}[(i)]
\item $r_R(M) = eR$ where $e \in \semicentral$.
\item $R / r_R(M) \iso (1-e)R$ as a ring and module where $(1-e) \in \rsemicentral$ and $(1-e)R = (1-e)R(1-e)$.
\item $M$ is a faithful s.Baer $R/r_R(M)$-module.
\item The ring $R/r_R(M)$ has the SSIP as a $R/r_R(M)$-module.
\end{enumerate}
\label{annihilator_ssip}
\end{thm}

\begin{proof}
$(i)$ Since $M$ is s.Baer, $r_R(M) = eR$ where $e \in \semicentral$. 

$(ii)$ Then $T := R / r_R(M)$ is ring and module isomorphic to $(1-e)R = (1-e)R(1-e)$ since $1-e \in \rsemicentral$. We denote an element of $T$ by $\overline{t}$. 

$(iii)$ It is routine to show $M$ is a faithful $T$-module. For $\varnothing \neq S \subseteq M$, we have $r_R(S) = e_sR$ for $e_s = e_s^2 \in R$. Now for $\overline{e_s}T \subseteq T$, $S\overline{e_s}T = S(e_sR + r_R(M)) = Se_sR = 0 \Rightarrow \overline{e_s}T \subseteq r_T(S)$. If $\overline{0} \neq \overline{a} \in r_{T}(S)$, $S\overline{a} = S(a + r_R(M))=Sa = 0$ implies $a \in e_sR$. Hence $M$ is a faithful s.Baer $T$-module. 

$(iv)$ Let $\{ \overline{e}_\alpha \}_{\alpha \in A}$ be an arbitrary family of idempotents of $T$ and let $S = \bigcup_A M(\overline{1}-\overline{e}_\alpha) \subseteq M$. Since $M_T$ is s.Baer, $r_T(S) = \overline{e}T$ for some $\overline{e}=\overline{e}^2 \in T$. Observe for $0 \neq \overline{a} \in r_T(M(\overline{1}-\overline{e}_\alpha))$, $M(\overline{1}-\overline{e}_\alpha)\overline{a} = 0$ implies $(\overline{1}-\overline{e}_\alpha)\overline{a} = 0$, since $M_T$ is faithful. So $r_T(M(\overline{1} - \overline{e}_\alpha)) = \overline{e}_\alpha T$. Hence, $\bigcap_{A} \overline{e}_\alpha T = r_T(S) = \overline{e}T$ for some $\overline{e}=\overline{e}^2 \in T$. Therefore $T_T$ has the SSIP.
\end{proof}


A ring $R$ is Baer if and only if $R$ is right Rickart with the SSIP (see \cite[Proposition 2.22]{RR}).
With $M = R$, this is an illustration of Theorem \ref{annihilator_ssip} while also providing motivation for the following theorem.



\begin{thm} For $0 \neq M \in \mathcal{M}od_R$, consider the following: \label{ssip}
\begin{enumerate}[(i)]
\item $M$ is s.Rickart and $R_R$ has the SSIP.
\item $M$ is s.Baer.
\item $Z(M)=0$ and $r_R(S) \ess eR$ for any $\varnothing \neq S \subseteq M$.
\end{enumerate} 
Then $(i) \Rightarrow (ii) \iff (iii)$. Additionally, if $M$ is faithful, then $(i) \iff (ii)$.
\end{thm}

\begin{proof}

$(i)\Rightarrow(ii)$ Let $\varnothing \neq S \subseteq M$. Then $r_R(s) = e_sR$ where $e_s = e_s^2$ for all $s \in S$. So $r_R(S) = \bigcap_S r_R(s) = \bigcap_S e_s R = eR$ for some $e = e^2 \in R$ since $R_R$ has the SSIP. Thus, $M$ is s.Baer.

$(ii)\Rightarrow (iii)$ This implication is immediate.

$(iii)\Rightarrow (ii)$ Let $\varnothing \neq S \subseteq M$ and $S \neq \{ 0 \}$. There exists $e=e^2 \in R$ such that $r_R(S) \ess eR$. Then $e^{-1}r_R(S) \ess R_R$ and $s e^{-1}r_R(S) = 0$ for each $s \in S$. Thus, $Z(M) \neq \{0\}$, a contradiction. Therefore, $r_R(S) = eR$ for some $e=e^2 \in R$. Hence, $M$ is s.Baer.

From Example \ref{counterexample}, we see that $(ii) \nRightarrow (i)$, in general. However, if $M$ is faithful, then Theorem \ref{annihilator_ssip} yields $(ii) \Rightarrow (i)$.
\end{proof}

The following theorem generalizes a well known result of Small \cite[Theorem 1]{Sm} that an orthogonally finite right Rickart ring is Baer.


\begin{thm} Let $R$ be orthogonally finite. Then $M_R$ is s.Rickart if and only if $M_R$ is s.Baer. \label{orthog_finite}
\end{thm}

\begin{proof}
$(\Rightarrow)$ Let $\varnothing \neq S \subseteq M$ and $0 \neq s \in S$. Then $r_R(S) \subseteq r_R(s) = eR$ for some $e=e^2 \in R$. The set $\mathfrak{X} := \{ r_R(X) \mid X \subseteq S, |X| < \infty \}$ is a poset under set inclusion. Note by \cite[Theorem 2.4]{LC}, that $r_R(X) \leq^{\oplus} R$ for each $X \in \mathfrak{X}$. Since $R$ is orthogonally finite, $R$ has DCC on right direct summands \cite[Proposition 1.2.13]{BPR} hence $\exists Y \in \mathfrak{X}$ such that $r_R (Y)$ is minimal in $\mathfrak{X}$ and $r_R (Y) = cR$ for some $c=c^2 \in R$. Now observe $r_R(S) \subseteq r_R(Y)$ and $r_R(S) = \bigcap_{S} r_R(s)$. If $s \in Y$, $r_R(s) \supseteq r_R (Y)$. If $s \in S - Y, r_R( Y \cup \{ s \} ) = r_R(Y) \cap r_R(s) = cR \cap eR = dR$ for some $d=d^2 \in R$ since $| Y \cup \{ s \} | < \infty$. Furthermore $r_R ( Y \cup \{ s \} ) \subseteq r_R ( Y ) \Rightarrow dR = cR$ since $r_R (Y)$ is minimal. Hence $cR \cap eR = cR \Rightarrow cR \subseteq eR = r_R(s)$. Thus $r_R(S) \subseteq r_R(Y) = cR \subseteq \bigcap_{S} r_R(s)  = r_R(S)$.

$(\Leftarrow)$ This implication is clear.
\end{proof}

Driven by Proposition \ref{finite_dsum} and Example \ref{counterexample}, we seek conditions for an arbitrary direct sum or direct product of s.Baer modules to be s.Baer. 


\begin{thm} \label{dsums}
Let $\{ M_\alpha \}_{\alpha \in A}$ be an indexed set of $R$-modules. Consider the following: 
\begin{enumerate}[(i)]
	\item $M_\alpha$ is a s.Baer $R$-module for every $\alpha \in A$.
	\item $\bigoplus_{A} M_\alpha$ is a s.Baer $R$-module.
	\item $\prod_{A} M_\alpha$ is a s.Baer $R$-module.
\end{enumerate}
Then $(iii) \iff (ii) \Rightarrow (i)$. Additionally if $R_R$ has the SSIP or is orthogonally finite, $(i) \Rightarrow (ii)$.
\end{thm}

\begin{proof} In general, we have $(iii) \Rightarrow (ii) \Rightarrow (i)$ by Lemma \ref{hered}.

$(ii) \Rightarrow (iii)$ Let $\varnothing \neq X \subseteq \prod_A M_\alpha$. For each $x \in X$, define $x_\alpha : A \rightarrow \bigcup_A M_\alpha$ where $x_\alpha (i) = x(\alpha)$ when $\alpha = i$ and 0 elsewhere. Then $x_\alpha \in \bigoplus_A M_\alpha$ and the set $S_x := \{ x_\alpha \mid \alpha \in A \} \subseteq \bigoplus_A M_\alpha$. Observe $r_R(x) = r_R(S_x)$. Then $r_R (X) = r_R ( \bigcup_X \{ x \} ) = r_R ( \bigcup_X S_x ) = eR$ since $\bigoplus_A M_\alpha$ is s.Baer.

$(i) \Rightarrow (ii)$ Assume $R_R$ has the SSIP or $R$ is orthogonally finite. From \cite[Theorem 2.7]{LC}, $\bigoplus_A M_\alpha$ is s.Rickart. Thus Theorems 1.11 and 1.12 give us that $\bigoplus_A M_\alpha$ is s.Baer. 

\end{proof}


\begin{cor} \label{baercor}
For a ring R, the following are equivalent:
\begin{enumerate}[(i)]
\item R is a Baer ring.
\item R is a right Rickart ring and every direct product of s.Rickart R-modules is a s.Rickart R-module.
\item $R$ is right Rickart and every direct product of s.Rickart R-modules is an s.Baer R-module.
\item $(\prod_{A} R)_R$ is a s.Rickart module for every set A.
\item  $(\prod_{A} R)_R$ is a s.Baer module for every set A.
\item Every submodule of a free R-module is a s.Baer R-module.
\item Every projective R-module is a s.Baer R-module.
\item Every torsionless R-module is a s.Baer R-module.
 
\end{enumerate}
\end{cor}
\begin{proof} (i),(ii), and (iv) are equivalent by \cite[Theorem 2.9]{LC}. 

$(ii)\Rightarrow (iii)$ Since $(i) \iff (ii)$, Theorem \ref{annihilator_ssip} yields that $R_R$ has the SSIP. Now this implication follows from Theorem \ref{ssip}.

$(iii) \Rightarrow (v)$ The proof of this implication is clear. 

$(v) \Rightarrow (vi)$ This implication follows from Lemma \ref{hered} and the fact that every free module is a submodule of some direct product of copies of $R$. 

$(vi) \Rightarrow (vii)$ This follows easily from the hypothesis. 

$(vii) \Rightarrow (viii)$ By assumption, we know that $R_R$ is s.Baer. Then $(\prod_A R)_R$ is s.Baer by Proposition \ref{direct_prod_same}. Given a torsionless $R$-module $M$, $M \hookrightarrow (\prod_A R)_R$. Therefore, $M_R$ is s.Baer by Lemma \ref{hered}. 

$(viii) \Rightarrow (i)$ $R$ is clearly torsionless and thus, by hypothesis, $R_R$ is s.Baer. Hence $R$ is Baer. 
\end{proof}


\begin{prop} \label{essnilp} 
Let R be a ring. If there is a nonzero $M \in \mathcal{M}od_R$ such that $M$ is s.Baer, then R cannot contain an essential nilpotent ideal.
\end{prop}
\begin{proof}
Let $0 \neq S_R \leq M_R$ and let $I \trianglelefteq R$ such that $I_R \ess R_R$, $I^n=0$, and $I^{n-1} \neq 0$ for some $n \in \mathbb{N}$. Let $0 \leq k \in \mathbb{N}$ be maximal with respect to $SI^k \neq 0$. Then $I \subseteq r_R (SI^k) = eR$ where $e=e^2 \in R $, a contradiction. 
\end{proof}


\begin{ex} The s.Baer concept is independent of the SIP and SSIP.
\end{ex}
\begin{enumerate}[(i)]
\item Every cyclic submodule of $M$ given in Example \ref{counterexample} is s.Baer, but $R_R$ does not have the SSIP.

\item Let $R = T_2 ( \Z )$ and $K = \left( \begin{smallmatrix} 0 & \Z \\ 0 & \Z \\ \end{smallmatrix} \right)$. $K_R$ is a s.Baer submodule that is not faithful and $R_R$ has the SSIP by Lemma \ref{triangular_ssip_lemma}.

\item Here we construct a s.Baer module M over a non-SIP ring $R$. Let $T$ be a commutative Baer ring that is not semisimple. By Proposition \ref{essnilp}, $T$ contains no essential nilpotent ideal. Since $T$ is not semisimple, there exists maximal ideal $P \subseteq T$ that is not a direct summand (hence $P \ess T$) and $P^2 \neq 0$. Now consider $R = \left( \begin{smallmatrix} T / P^2 & P / P^2 \\ 0 & T \\ \end{smallmatrix} \right)$. Observe that $ M_R  = \left( \begin{smallmatrix} 0 & 0 \\ 0 & T \\ \end{smallmatrix} \right)$ is s.Baer. Then for $\overline{0} \neq \overline{p} \in P / P^2$, $\left( \begin{smallmatrix} 0 & \overline{p} \\ 0 & 1 \\ \end{smallmatrix} \right)R \cap \left( \begin{smallmatrix} 0 & 0 \\ 0 & 1 \\ \end{smallmatrix} \right)R = \left( \begin{smallmatrix} 0 & 0 \\ 0 & P \\ \end{smallmatrix} \right)$ is not generated by an idempotent since $P \nleq^{\oplus} T$. Thus $R_R$ does not have the SIP; an example of such an $R$ is $\left( \begin{smallmatrix} \Z_4 & 2 \Z_4 \\ 0 & \Z \\ \end{smallmatrix} \right)$. Finally, note that $M \iso R / r_R(M)$ is the nonzero factor of $R_R$ with the SSIP guaranteed by Theorem \ref{annihilator_ssip}.

\item Motivated by Theorem \ref{dsums}, we show that if a direct product of a family of s.Baer $R$-modules is s.Baer then $R_R$ need not have the SSIP. Fix a prime $p \in \Z$. Let $\mathfrak{P}$ be the set of all prime numbers in $\Z$, $R = \left( \begin{smallmatrix} \Z_{p^2} & p \Z_{p^2} \\ 0 & \Z \\ \end{smallmatrix} \right)$, and $M_q = \left( \begin{smallmatrix} 0 & 0 \\ 0 & \Z \\ \end{smallmatrix} \right)$ for every $q \in \mathfrak{P}$. By Proposition \ref{direct_prod_same}, $\prod_{\mathfrak{P}} M_q$ is s.Baer, but $R_R$ does not have the SSIP (see $(iii)$).  
\end{enumerate}

\begin{defn} $M \in \modr$ is finitely idempotent faithful if, for each nontrivial idempotent $e \in R$ (i.e., $e\neq 0$ and $e \neq 1$), there exists a nonempty finite subset $S \subseteq M$ such that $r_R(S)\cap eR =0$.
\end{defn}

Examples of finitely idempotent faithful modules include modules $M_R$ such that $R_R \hookrightarrow M$. This leads us to a generalization of \cite[Corollary 2.6]{LC}.

\begin{thm} If $M \in \modr$ is s.Rickart and finitely idempotent faithful, then $R_R$ has the SIP. \label{sip_result}
\end{thm}

\begin{proof}
Let $e,f$ be nontrivial idempotents in $R$ and consider $eR \cap fR$. For $1-e,1-f \in R$, there exists nonempty finite subsets $S_1,S_2$ in $M$, respectively, such that $r_R(S_1) \cap (1-e)R = 0$ and $r_R(S_2) \cap (1-f)R = 0$. Now $eR \cap fR \subseteq r_R(S_1(1-e) \cup S_2(1-f)) = cR$ where $c=c^2 \in R$ since $S_1(1-e), S_2(1-f)$ are finite subsets \cite[Theorem 2.4]{LC}. Let $0 \neq a \in cR$ and observe that $S_1(1-e)a=0$ implies $(1-e)a \in r_R(S_1(1-e)) \cap (1-e)R$. Then $ (1-e)a = 0$ if and only if $a = ea$ if and only if $a \in eR$. Similarly, $S_2(1-f)a=0 $ implies $ a \in fR$. Hence, $eR \cap fR = cR$.
\end{proof}

\begin{cor} Let $M \in \modr$ be s.Rickart. Then $R_R$ has the SIP if any of the following hold:
\begin{enumerate}[(i)]
\item $R_R \hookrightarrow \bigoplus_{i=1}^{n} M_i$ where $M_i \iso M$ for $i=1,\ldots,n$ (e.g. $M$ is a generator in $\modr$).
\item There exists a monomorphism $h: \mathcal{I} (R)_{\mathcal{I} (R)} \hookrightarrow M_{\mathcal{I} (R)}$ where $\mathcal{I} (R) = \langle \{ e \in R \mid e = e^2 \} \rangle$.
\item For all $e = e^2 \in R$, there exists a nonempty finite $S \subseteq M$ such that $r_R(S) = eR$.
\end{enumerate}
\end{cor}

\begin{proof}
$(i)$ By \cite[Theorem 2.7]{LC}, $\bigoplus_{i=1}^n M_i$ is s.Rickart. Let $0 \neq e=e^2 \in R$ and consider $\iota : R \hookrightarrow \bigoplus_{i=1}^n M_i$. Then $r_R(\iota (1))  = 0$ and hence $r_R(\iota (1)) \cap eR = 0$. By Theorem \ref{sip_result}, $R_R$ has the SIP.

$(ii)$ Let $e \in R$ be a nontrivial idempotent. Then $n = h(e) \in M$ and $r_{\mathcal{I} (R)} (n) = (1-e) \mathcal{I} (R)$. Next $r_{\mathcal{I} (R)} (n) = r_R(n) \cap \mathcal{I} (R) = cR \cap \mathcal{I} (R)$ where $c = c^2 \in R$. Observe $c \in r_R(n) \cap \mathcal{I} (R)$ and hence $c = (1-e)i \in (1-e)R$ where $i \in \mathcal{I} (R)$. Thus $cR \subseteq (1-e)R$. Since $n(1-e) = 0 $, $ 1-e \in r_R(n) = cR $. Hence $ cR = (1-e)R $ and so $ r_R(n) \cap eR = 0 $. Therefore $R_R$ has the SIP by Theorem \ref{sip_result}.

$(iii)$ Let $e,f$ be idempotents in $R$ and let $S$ be a nonempty finite subset of $M$ such that $r_R(S) = eR$. Since $M$ is s.Rickart, $fR \cap eR = fR \cap r_R(S) \leq^{\oplus} R_R$ by \cite[Lemma 2.3]{LC}. Hence, $R_R$ has the SIP.
\end{proof}


\begin{prop} Let $M \in \modr$ be semisimple. Consider the following conditions:
\begin{enumerate}[(i)]
\item $M$ is nonsingular.
\item $M$ is projective.
\item $M$ is s.Rickart.
\item $M$ is s.Baer.
\end{enumerate}
Then $(i)$-$(iii)$ are equivalent. If $M$ has only finitely many homogeneous components, then $(i)$-$(iv)$ are equivalent.
\label{nonsingular_simple}
\end{prop}

\begin{proof}
$(i) \Rightarrow (ii)$ A simple nonsingular module is projective \cite[Proposition 1.24]{G}. Thus, a direct sum of simple nonsingular modules is nonsingular and projective.

$(ii) \Rightarrow (iii)$ Let $n \in M$. Then $nR$ is a direct summand of $M$ and, hence, is projective. Since all cyclic submodules are projective, $M$ is s.Rickart. 

$(iii) \Rightarrow (i)$ The part is immediate.

$(iii) \Rightarrow (iv)$ Without loss of generality, assume $M$ has two homogeneous components $C_1$ and $C_2$. Then there are simple submodules $S_1$ and $S_2$ such that $ C_1 = \bigoplus_I X_i$ where $X_i \iso S_1$ for all $i \in I$ and $C_2 = \bigoplus_A Y_\alpha$ where $Y_\alpha \iso S_2$ for all $\alpha \in A$. Let $\varnothing \neq N_i \subseteq S_i$ and let $0 \neq n_i \in N_i$. Then $r_R(n_i) = r_R(N_i) = eR$ for some $e = e^2 \in R$, since the $S_i$ are simple. From Proposition \ref{direct_prod_same}, $C_1$ and $C_2$ are s.Baer. Hence $M \iso S_1 \oplus S_2$ is s.Baer by Proposition \ref{finite_dsum}.

$(iv) \Rightarrow (i)$ This part is immediate.
\end{proof}

To further motivate our next result, recall the following: for a ring $R$ with $Soc(R_R) \neq 0$, $R$ is right primitive if and only if $R$ is prime.

\begin{thm}
\begin{enumerate}[(i)]
\item $R$ is a right primitive ring with $Soc(R_R) \neq 0$ if and only if there exists a faithful simple s.Baer $R$-module. 
\item Let $R$ be right primitive ring with $Soc(R_R) \neq 0$. Then $R_R$ has the SSIP and every faithful simple $R$-module is s.Baer.
\end{enumerate}
\label{primitive}
\end{thm} 

\begin{proof}
$(i) (\Rightarrow)$ A routine argument will show that $Z(R_R) = 0$ since $Soc(R_R) \ess R_R$. By assumption, there exists a faithful simple $R$-module $M$. For $0 \neq n \in M$, $M = nR \iso R/r_R(n)$. If $r_R(n) = eR$ for $e=e^2 \in R$, then $M \iso (1-e)R$ and $(1-e)R \subseteq Soc(R_R)$. Thus $M$ is a nonsingular simple module and, hence, is s.Baer by Proposition \ref{nonsingular_simple}. If $r_R(n) \nleq^{\oplus} R_R$ then $r_R(n) \ess R_R$ since $r_R(n)$ is a maximal right ideal. Now $Soc(R_R) \subseteq r_R(n)$. Since $M$ is faithful, $r_R(n)$ cannot contain a nonzero ideal contrary to $Soc(R_R) \neq 0$. Thus $r_R(n) \dsum R_R$ implies $M$ is s.Rickart. Again Proposition \ref{nonsingular_simple} yields $M$ is s.Baer.

$(\Leftarrow)$ Clearly, $R$ is right primitive. Since $M$ is simple and s.Rickart, $M = nR \iso R/r_R(n)$ for $0 \neq n \in M$ where $r_R(n) \dsum R_R$ is a maximal right ideal. Thus there is an $aR \leq R_R$ such that $aR \oplus r_R(n) = R$ where $aR$ is a minimal right ideal. Hence $Soc(R_R) \neq 0$.

$(ii)$ By $(i)$, there exists a faithful simple s.Baer $R$-module, and, by Theorem \ref{ssip}, $R_R$ has the SSIP. Using the argument in $(\Rightarrow)$ of $(i)$, we see that every faithful simple $R$-module is s.Baer.
\end{proof}


\begin{prop} \label{indecomposable}
Let $R_R$ be indecomposable. Then $M_R$ is s.Baer if and only if $M_R$ is s.Rickart if and only if $\ r_R(m)=0$ for all $0 \neq m \in M_R$. 
\end{prop}

\begin{proof}
Clearly, $M_R$ is s.Baer implies $M_R$ is s.Rickart implies $r_R(m)=0$ for all $0 \neq m \in M$. Let us assume $r_R(m) = 0$ for all $ 0 \neq m \in M$ and let $\varnothing \neq S \subseteq M$. If $S = \{ 0 \}$, then $r_R(S) = R$. Now suppose $S \neq \{ 0 \}$. Then $r_R(S) = \bigcap_{S} r_R(s) = 0$. Therefore $M_R$ is s.Baer.
\end{proof}


Consider $\mathbb{Z}_{\mathbb{Z}} \rightarrow  (\mathbb{Z}_4)_{\mathbb{Z}}$. We can readily observe that the s.Baer property is not preserved by arbitrary module homomorphisms. In spite of this handicap, under certain conditions the s.Baer property integrates well with the $Hom$ functor.

\begin{prop} Let $N,M \in \modr$ and let $M$ be s.Baer.
\begin{enumerate}[(i)]
\item If $R$ is commutative and every $0 \neq h \in Hom_R(M,N)$ is a monomorphism (e.g., $N$ is monoform and $M \leq N$), then $Hom_R(M,N)$ is a s.Baer $R$-module and $r_R(Hom_R(M,N)) = r_R(M)$.

\item Let $S$ be a ring. If $N,M$ are $(S,R)$-bimodules, then $Hom_S(N,M)$ is a right s.Baer $R$-module. In particular, if $M = R = S$ is a Baer ring, then the dual module $N^* = Hom(_SN$,$_SS)$ is a right s.Baer $S$-module.
\end{enumerate}
\end{prop}

\begin{proof}
$(i)$ First we will show $r_R(S) = r_R(h(S))$ for $\varnothing \neq S \subseteq M$. Observe $r_R(S) \subseteq r_R(h(S))$. Let $a \in r_R(h(S))$ and consider $h(S)a = h(Sa) = 0$. Then $Sa = 0 $ implies $ a \in r_R(S)$. So $r_R(S) = r_R( h(S) )$. Now for $0 \neq h \in Hom_R(M,N)$, $r_R(h) = \{ a \in R \mid h \ast a = 0 \} = \{a \in R \mid h(m) a = 0 \  \forall m \in M \} = r_R(Im(h)) = r_R(h(M)) = r_R(M) = eR$ for some $e = e^2 \in R$ since $M$ is s.Baer. Then for any $\varnothing \neq H \subseteq Hom_R(M,N)$, $r_R(H) = \bigcap_H r_R(h) = r_R(M) = eR$. In particular, $r_R(Hom_R(M,N)) = eR$.

$(ii)$ Let $\varnothing \neq H \subseteq Hom_S(N,M)$. Then $r_R(H) = \bigcap_H r_R(h) = \bigcap_H r_R(Im(h)) = r_R(\sum_H Im(h)) = eR$ for some $e = e^2 \in R$ since $\sum_H Im(h) \subseteq M$ and $M_R$ is s.Baer. Thus $Hom_S(N,M)$ is a s.Baer right $R$-module. Now supposing $M = R = S$ is a Baer ring, then $N^* = Hom(_SN$,$_SS)$ is a torsionless right $S$-module \cite[p. 145]{L}. Hence by Corollary \ref{baercor} (vii), $N^*$ is a right s.Baer $S$-module.
\end{proof}

\begin{prop} \label{essential_hom}
Let $K,M,T \in \modr$ with $T \ess M$. If K is s.Rickart (s.Baer) and $Hom_R (T,K)=0$, then $Hom_R(M,K)=0$.
\end{prop}

\begin{proof}
Suppose, to the contrary, that $Hom_R(M,K) \neq 0$ (i.e., $ \exists 0 \neq h \in Hom_R(M,K)$). Consider the short exact sequence
$0 \rightarrow ker (h) \rightarrow M \rightarrow M/ker(h) \rightarrow 0$, where $M/ker(h) \simeq Im(h) \leq K$. Since $K \in \sR \ (\sB )$, by Lemma \ref{hered}, $M/ker(h) \in \sR \ ( \sB )$. Also note $ker(h) \ess M$ since $h(T) = 0$. For $0 \neq x + ker(h) \in M/ker(h)$, $\frac{xR+ker(h)}{ker(h)} \iso \frac{xR}{xR\cap ker(h)} \in \sR \ (\sB )$ and projective. Let $\iota:\frac{xR}{xR\cap ker(h)} \rightarrow \frac{xR}{xR\cap ker(h)}$ be the identity map and $g:xR \rightarrow \frac{xR}{xR\cap ker(h)}$ be the natural map. Since $\frac{xR}{xR\cap ker(h)}$ is projective, there exists $\overline{f}: \frac{xR}{xR\cap ker(h)} \rightarrow xR$ such that $g \overline{f} = \iota$ and $g$ splits $xR$. Hence, $xR = Im(\overline{f}) \oplus ker(g)$ where $ker(g) = xR \cap ker(h)$. Now $ker(h) \cap Im(\overline{f}) = 0$ contrary to $ker(h)$ being essential. Thus we have that $Hom_R(M,K) = 0$.
\end{proof}


\begin{ex} \label{1stexample}
The following examples provide further motivation for investigating s.Baer modules.
\end{ex}
\begin{enumerate}[(i)]

\item By \cite[Lemma 3]{K}, every $AW^*$-module over a commutative $AW^*$-algebra is a s.Baer module. 

\item By \cite[p. 352]{BM}, any algebraically finitely generated $C^*$-module is projective in the sense of pure algebra. Thus every algebraically finitely generated $C^*$-module over an $AW^*$-algebra (Rickart $C^{*}$-algebra) $A$ is an s.Baer (s.Rickart) $A$-module by Corollary \ref{baercor} (vii).

\item Recall \cite[Theorem 2.1]{CK}: A ring $R$ is a right nonsingular right extending ring if and only if $R$ is a right cononsingular Baer ring. Thus Theorems \ref{ssip} and \ref{extend} tell us that over a right nonsingular right extending ring that a module is s.Baer if and only if it is nonsingular.

\item Let $R$ be a Baer ring. Any torsionless $R$-module $M$ is s.Baer by Lemma \ref{hered} and Corollary \ref{baercor}. In this instance, the class of torsionless modules is contained properly in the class of s.Baer modules. For instance, we can take $\mathbb{Q}_\mathbb{Z}$ which is an s.Baer module that is not torsionless since $Hom_\mathbb{Z}(\mathbb{Q}, \mathbb{Z}) = 0$.

\item Here we show that even a nonsingular module over a SSIP ring cannot capture the s.Baer property. Let $M = R = T_2(\Z)$. Then $R_R$ has the SSIP by Lemma \ref{triangular_ssip_lemma} and is nonsingular. However $R$ is not Baer \cite[p. 16]{K1}. 

\end{enumerate}


\begin{thm} \label{semisimple}
The following are equivalent:
\begin{enumerate}[(i)]
\item $R_R$ is semisimple.
\item Every $R$-module is a s.Baer $R$-module.
\item Every $R$-module is a s.Rickart $R$-module.
\item Every cyclic $R$-module is a s.Baer $R$-module.
\end{enumerate}
\end{thm}

\begin{proof}
$(i)\Rightarrow (ii)$ Since $R_R$ is semisimple, $R_R$ has SSIP. For $m \in M_R$, $r_R(m) \leq^\oplus R_R$. Thus $M$ is s.Baer by Theorem \ref{ssip}. 

$(ii) \Rightarrow (iii)$ and $(ii) \Rightarrow (iv)$ are clear. 

$(iii)\Rightarrow (i)$ This follows from \cite[Proposition 2.10]{LC}. 

$(iv)\Rightarrow (iii)$ Let $M \in \modr$. By assumption, $mR$ is s.Baer for all $m \in M$, and hence for $m \in mR$,  we have that $r_R(m) = eR$ for some $e=e^2 \in R$.
\end{proof}

Recall from \cite[Definition 1.3]{ABT}, a module $M$ is $\mathcal{G}$-$extending$ if and only if for each $X \leq M$, there exists $D \leq^{\oplus} M$ such that $X \cap D \ess D$ and $X \cap D \ess X$. Note that every extending module is $\mathcal{G}$-extending. See \cite{ABT} for further examples and details. Since every s.Rickart module is nonsingular it is natural to ask: \textit{When is every nonsingular module a s.Rickart or s.Baer module?} Our next result gives a partial answer to this question.


\begin{thm} \label{extend}
 Let $R_R$ be a $\mathcal{G}$-extending module. 
 \begin{enumerate}[(i)]
 \item $M_R$ is s.Rickart if and only if $Z(M) = 0$.
 \item If $R_R$ has the SSIP or $R$ is orthogonally finite, then $M_R$ is s.Baer if and only if $Z(M) = 0$.
 \end{enumerate}
\end{thm}

\begin{proof} $(i)$ $(\Rightarrow )$ This is clear. $(\Leftarrow )$ Let $0 \neq m \in M$ and consider $r_R(m)$ in $R$. Since $R_R$ is $\mathcal{G}$-extending, there exists an $e=e^2 \in R$ such that $r_R(m) \cap eR \leq^{ess} r_R(m)$ and $r_R(m) \cap eR \leq^{ess} eR$. Denote $X = r_R(m) \cap eR$ and consider $L = e^{-1}X = \{ r \in R \ | \ er \in X \}$. Now, $L_R \leq^{ess} R_R$. Therefore, $meL = 0$ so $me \in Z(M) = \{ 0 \}$. Thus $e \in r_R(m)$, and $r_R(m) = eR$.  

$(ii)$ Combining Theorems \ref{ssip} and \ref{orthog_finite} and part $(i)$, we obtain the result.
\end{proof}

\section{Links with Projectivity} \label{projectivity_section}


From Corollary \ref{baercor}, we know that if $R$ is a Baer ring then every projective module is s.Baer. That result and the first result of this section motivate the following question: \textit{When is every s.Baer (s.Rickart) module projective?} Observe that even for a Baer ring not every s.Baer module is projective (e.g., $\mathbb{Q}_\mathbb{Z}$ is s.Baer, by Proposition \ref{indecomposable}, but $\mathbb{Q}_\mathbb{Z}$ is not projective). This question seems both natural and interesting since in every s.Rickart (hence s.Baer) module all cyclic submodules are projective. In this section, we answer the question when $R$ is a right cononsingular ring. We also determine a class of generalized triangular matrix rings satisfying the condition that every s.Baer (s.Rickart) module is projective.


\begin{thm} \label{essprojective}
Every s.Rickart module is an essential extension of a projective module.
\end{thm}

\begin{proof} 
Let $M_R$ be a s.Rickart module and let $\{ m_\gamma R \}_{\gamma \in \Gamma}$ be the family of all cyclic submodules of M indexed by the set $\Gamma$. Consider the family of sets $\mathcal{P} = \{ \Omega_i \ | \ i \in I \}$ where $ \Omega_i \subseteq \Gamma$ for all $i \in I$ and $\sum_{\Omega_i} m_\omega R$ is a direct sum. Observe that $\mathcal{P}$ is a nonempty poset ordered by set inclusion. Consider an arbitrary chain $\mathcal{C}$ in $\mathcal{P}$ where $\mathcal{C} = \{ \Omega_\lambda \ | \ \lambda \in \Lambda, \Lambda \subseteq I \}$. Let $\Omega = \cup_{\Lambda} \Omega_\lambda$. We will show $\Omega \in \mathcal{P}$. Suppose $\sum_{\Omega} m_\omega r_\omega = 0$, where $m_\omega r_\omega \in m_\omega R$ and $m_\omega r_\omega \neq 0$ for finitely many $\omega$. This finite collection $\{\omega \in \Omega \mid m_\omega r_\omega \neq 0 \}$ will lie in some $\Omega_\lambda$, but $\Omega_\lambda \in \mathcal{P}$. Hence all $m_\omega r_\omega = 0$ and thus we have that $\Omega \in \mathcal{P}$. By Zorn's lemma, there exists a set $\Omega_\Upsilon \subseteq \mathcal{P}$ maximal with respect to $\sum_{\Omega_\Upsilon} m_\epsilon R$ being a direct sum. Thus $M$ contains a maximal direct sum of cyclic submodules. Furthermore each cyclic is projective since $M$ is s.Rickart, hence $\bigoplus_{\Omega_\Upsilon} m_\epsilon R$ is projective. Let us denote $\bigoplus_{\Omega_\Upsilon} m_\epsilon R = K$ and show it is essential in M. Suppose, to the contrary, there is $0 \neq S_R \leq M_R$ where $K \cap S = 0$. Then $K \cap sR = 0$ for all $s \in S$. Thus $K + sR$ is a direct sum which contradicts the maximality of $\Omega_\Upsilon$.
\end{proof}

Thus far we have seen that the SSIP is a useful tool for gauging which modules may be s.Baer. For the remainder of this section and throughout the next, the following lemma will prove to be valuable to us.
 

\begin{lem} \label{triangular_ssip_lemma}
Let $R =  \left( \begin{smallmatrix} A & M \\ 0 & C \\ \end{smallmatrix} \right)$ where $R$ is a generalized triangular matrix ring such that $A$ and $C$ are rings and $M$ an $(A,C)$-bimodule. Further suppose that $A_A$ and $C_C$ are indecomposable. Then $r_C(m) = 0$ for all $0 \neq m \in M$ if and only if $R_R$ has the SSIP.
\end{lem}

\begin{proof}
$(\Leftarrow)$ First, note that if $M = 0$, then $R_R$ has the SSIP. Thus, let $M \neq 0$. Observe that all nontrivial idempotents are of the form $ \left( \begin{smallmatrix} 1 & m \\ 0 & 0 \\ \end{smallmatrix} \right)$ and $ \left( \begin{smallmatrix} 0 & n \\ 0 & 1 \\ \end{smallmatrix} \right)$ where $m,n \in M$. Assume $R_R$ has the SSIP and let $0 \neq n \in M$. Then $\left( \begin{smallmatrix} 0 & n \\ 0 & 1 \\ \end{smallmatrix} \right) R \cap \left( \begin{smallmatrix} 0 & 0 \\ 0 & 1 \\ \end{smallmatrix} \right) R = \left( \begin{smallmatrix} 0 & 0 \\ 0 & r_C(n) \\ \end{smallmatrix} \right).$ Since $R_R$ has the SSIP, $r_C(n) = 0$.

$(\Rightarrow )$ Conversely, assume $r_C(m) = 0$ for all $0 \neq m\in M$. Now observe that $ \left( \begin{smallmatrix} 1 & m \\ 0 & 0 \\ \end{smallmatrix} \right) R = \left( \begin{smallmatrix} A & M \\ 0 & 0 \\ \end{smallmatrix} \right)
$ and $\left( \begin{smallmatrix} 0 & n \\ 0 & 1 \\ \end{smallmatrix} \right)R = \{ \left( \begin{smallmatrix} 0 & nc \\ 0 & c \\ \end{smallmatrix} \right) \mid c \in C \}$. A routine argument will show $\left( \begin{smallmatrix} 0 & n \\ 0 & 1 \\ \end{smallmatrix} \right)R \cap \left( \begin{smallmatrix} 1 & m \\ 0 & 0 \\ \end{smallmatrix} \right) R = 0_R$. Next consider $x \in \left( \begin{smallmatrix} 0 & n \\ 0 & 1 \\ \end{smallmatrix} \right) R \cap \left( \begin{smallmatrix} 0 & k \\ 0 & 1 \\ \end{smallmatrix} \right) R $ for $n \neq k \in M$. Then $x =\left( \begin{smallmatrix} 0 & nd \\ 0 & d \\ \end{smallmatrix} \right) = \left( \begin{smallmatrix} 0 & kc \\ 0 & c \\ \end{smallmatrix} \right)$ for some $d,c \in C$. If $c$ or $d$ is zero, we are done. Assume that $c \neq 0$ and $d \neq 0$. If $n \neq 0$ then $c = d$ and $nc = kc$, i.e., $(n-k)c = 0$. Since $c \neq 0$, $n - k = 0$ implies $n = k$, a contradiction. If $n = 0$ then $kc = 0$. Since $k \neq 0$, $c = 0$, a contradiction. Finally, if we consider an arbitrary intersection of nonzero direct summands, it necessarily reduces to the cases presented above. Therefore $R_R$ has the SSIP.
\end{proof}

Recall that in any domain every principal right ideal is projective. However $T_2(\Z)$ is not a right Rickart ring, so not every principal right ideal is projective. Thus it is natural to ask: \textit{If $C$ is a domain, which principal right ideals of $T_2(C)$ are projective?}


\begin{thm} \label{gen_triangular_matrix_ring_thm}
Let $R =  \left( \begin{smallmatrix} A & M \\ 0 & C \\ \end{smallmatrix} \right)$ and $K = \left( \begin{smallmatrix} 0 & M \\ 0 & C \\ \end{smallmatrix} \right)$, where $R$ is a generalized triangular matrix ring such that $A$ and $C$ are rings and $M$ is an $(A,C)$-bimodule.
\begin{enumerate}[(i)]
\item $K_R$ is s.Rickart (s.Baer) if and only if $M_C$ and $C_C$ are s.Rickart (s.Baer).
\item Assume $A$ and $C$ are domains and for each $0 \neq m \in M$, $l_A(m) = 0$ and $r_C(m)=0$. Then $r_R( \left( \begin{smallmatrix} a & m \\ 0 & c \\ \end{smallmatrix} \right) ) \neq eR$ for any $e = e^2 \in R$ if and only if $a \neq 0, m \neq 0, c = 0$ and $mC \nsubseteq aM$. Moreover, if $R_R$ is not s.Rickart, then $K$ is the largest s.Rickart submodule of $R$, and all s.Rickart $R$-modules are s.Baer.
\end{enumerate}
\end{thm}

\begin{proof}

$(i)$ Note that $K_R$ is s.Rickart (s.Baer) if and only if $K_C$ is s.Rickart (s.Baer). The s.Rickart part follows from \cite[Theorem 2.7]{LC} and the s.Baer part from Proposition \ref{finite_dsum}.

$(ii)$ Recall that all idempotents of $R$ are of the form $0_R$, $\left( \begin{smallmatrix} 0 & x \\ 0 & 1  \\ \end{smallmatrix} \right)$, $\left( \begin{smallmatrix} 1 & y \\ 0 & 0  \\ \end{smallmatrix} \right)$, or $1_R$ for $x,y \in M$. We examine $r_R( \left( \begin{smallmatrix} a & m \\ 0 & c \\ \end{smallmatrix} \right) )$ on a case-by-case basis.\\
\textbf{Case 1:} $m = 0$.
	\begin{enumerate}[(I)]
	\item  $a = 0$ and $c \neq 0$ if and only if $r_R ( \left( \begin{smallmatrix} a & m \\ 0 & c \\ \end{smallmatrix} \right) ) = \left( \begin{smallmatrix} 1 & 0 \\ 0 & 0 \\ \end{smallmatrix} \right) R$.
	\item Suppose $a \neq 0$.
		\begin{enumerate}[(i)]
		\item  $c = 0$ if and only if $r_R ( \left( \begin{smallmatrix} a & m \\ 0 & c \\ \end{smallmatrix} \right) ) = \left( \begin{smallmatrix} 0 & 0 \\ 0 & 1 \\ \end{smallmatrix} \right) R$.
		\item  $c \neq 0$ if and only if $r_R ( \left( \begin{smallmatrix} a & m \\ 0 & c \\ \end{smallmatrix} \right) ) = \{ 0_R \}$.
		\end{enumerate}	
	\end{enumerate}
\textbf{Case 2:} $m \neq 0$.
	\begin{enumerate}[(I)]

	\item $a \neq 0$ and $c \neq 0$ if and only if $r_R ( \left( \begin{smallmatrix} a & m \\ 0 & c \\ \end{smallmatrix} \right) ) = \{ 0_R \}$.

	\item \textbf{Claim:} \textit{Assume $a\neq 0$ and $c = 0$. Then $r_R( \left( \begin{smallmatrix} a & m \\ 0 & c \\ \end{smallmatrix} \right)) = eR$ for some $e = e^2 \in R$ if and only if $mC \subseteq aM$.}\\
	\textit{Proof of Claim.} Assume that $r_R( \left( \begin{smallmatrix} a & m \\ 0 & 0 \\ \end{smallmatrix} \right) ) = eR$ for some $e = e^2 \in R$. Then $e = \left( \begin{smallmatrix} 0 & n \\ 0 & 1 \\ \end{smallmatrix} \right)$ for some $n \in M$. Hence $m = a(-n)$, so $mC \subseteq aM$.
	
Conversely, assume $mC \subseteq aM$. Then $m = ak$ for some $k \in M$. Suppose $\left( \begin{smallmatrix} a & m \\ 0 & 0 \\ \end{smallmatrix} \right) \left( \begin{smallmatrix} 0 & n \\ 0 & d \\ \end{smallmatrix} \right) = \left( \begin{smallmatrix} 0 & an + md \\ 0 & 0 \\ \end{smallmatrix} \right) = \left( \begin{smallmatrix} 0 & 0 \\ 0 & 0 \\ \end{smallmatrix} \right)$. Thus $0 = an + md = an + akd = a(n +kd)$ and so $n = -kd$. Therefore $r_R( \left( \begin{smallmatrix} a & m \\ 0 & 0 \\ \end{smallmatrix} \right) ) = \left( \begin{smallmatrix} 0 & -k \\ 0 & 1 \\ \end{smallmatrix} \right) R$. 

	\item $a=0$ and $c=0$ if and only if $r_R ( \left( \begin{smallmatrix} a & m \\ 0 & c \\ \end{smallmatrix} \right) ) = \left( \begin{smallmatrix} 1 & 0 \\ 0 & 0 \\ \end{smallmatrix} \right) R$.
	
	\item $a=0$ and $c \neq 0$ if and only if $r_R ( \left( \begin{smallmatrix} a & m \\ 0 & c \\ \end{smallmatrix} \right) ) = \left( \begin{smallmatrix} 1 & 0 \\ 0 & 0 \\ \end{smallmatrix} \right) R$.
	\end{enumerate}
Note that we have exhausted all cases. Therefore $r_R( \left( \begin{smallmatrix} a & m \\ 0 & c \\ \end{smallmatrix} \right) ) \neq eR$ for any $e = e^2 \in R$ if and only if $a \neq 0, m \neq 0, c = 0$ and $mC \nsubseteq aM$.

Assume $R_R$ is not s.Rickart. To see that $K_R$ is the largest s.Rickart submodule of $R$, let $X$ be a s.Rickart submodule of $R_R$. If $X \subseteq K$, we are done. Otherwise there exists $y = \left( \begin{smallmatrix} a & m \\ 0 & c \\ \end{smallmatrix} \right)$, where $a \neq 0$. Observe that each element of $yR$ has its right annihilator generated by an idempotent. Since $R_R$ is not s.Rickart, there exists $\left( \begin{smallmatrix} b & n \\ 0 & 0 \\ \end{smallmatrix} \right) \in R$ where $b \neq 0, n \neq 0$ and $nC \nsubseteq bM$. Then $ \left( \begin{smallmatrix} a & m \\ 0 & c \\ \end{smallmatrix} \right) \left( \begin{smallmatrix} b & n \\ 0 & 0 \end{smallmatrix} \right) = \left( \begin{smallmatrix} ab & an \\ 0 & 0 \end{smallmatrix} \right)$. By the first part of the proof of $(ii)$, $anC \subseteq abM$, then $an = abk$ for some $k \in M$. Hence $a(n - bk) = 0$ implies $n = bk$, a contradiction. Therefore $X \subseteq K$. Lemma \ref{triangular_ssip_lemma} and Theorem \ref{ssip} yield that all s.Rickart modules are s.Baer.
\end{proof}

\begin{cor} Let $R$ be as in Theorem \ref{gen_triangular_matrix_ring_thm}. Assume $A$ is a subring of $C$, $M=C$, and $C$ is a domain.
\begin{enumerate}[(i)]
\item $R$ is a right Rickart ring if and only if $R$ is a Baer ring if and only if $C$ is a division ring.
\item If $R$ is not a right Rickart ring, then every (finitely generated) s.Rickart submodule of $R_R$ is projective if and only if every (finitely generated) s.Baer submodule of $R_R$ is projective if and only if $C$ is right (semi)hereditary.
\end{enumerate}
\label{gen_triangular_matrix_cor}
\end{cor}

\begin{proof}
$(i)$ Lemma \ref{triangular_ssip_lemma} and Theorem \ref{ssip} give us that $R_R$ is s.Rickart if and only if $R$ is Baer. If $C$ is a division ring, it is well known that $T_2(C)$ is a Baer ring \cite[p. 16]{K1}. Since $R$ has the same idempotents as $T_2(C)$, $R$ is Baer \cite[p. 16]{K1}. Now assume that $R$ is a Baer ring. Let $0 \neq c \in C$. Then $0 \neq l_R( \left( \begin{smallmatrix} 0 & 1 \\ 0 & c \\ \end{smallmatrix} \right) )$. Hence there exists  $\left( \begin{smallmatrix} 1 & x \\ 0 & 0 \\ \end{smallmatrix} \right) \in R$ such that $l_R( \left( \begin{smallmatrix} 0 & 1 \\ 0 & c \\ \end{smallmatrix} \right) ) = R \left( \begin{smallmatrix} 1 & x \\ 0 & 0 \\ \end{smallmatrix} \right)$. So $ 0 = \left( \begin{smallmatrix} 1 & x \\ 0 & 0 \\ \end{smallmatrix} \right) \left( \begin{smallmatrix} 0 & 1 \\ 0 & c \\ \end{smallmatrix} \right)$ implies $1 + xc = 0$. Then $c$ is invertible and, thus, $C$ is a division ring. 

$(ii)$ The equivalence of the first two statements follow from Lemma \ref{triangular_ssip_lemma} and Theorem \ref{ssip}. From Theorem \ref{gen_triangular_matrix_ring_thm}, $\left( \begin{smallmatrix} 0 & C \\ 0 & C \\ \end{smallmatrix} \right)$ is the largest s.Rickart $R$-submodule contained in $R$.
Moreover, Lemma \ref{hered} establishes that each of its submodules is also s.Rickart.

Now assume that every (finitely generated) s.Rickart submodule of $R_R$ is projective. Then all (finitely generated) submodules of $\left( \begin{smallmatrix} 0 & 0 \\ 0 & C \\ \end{smallmatrix} \right)$ are projective. Therefore $C$ is right (semi)hereditary.

Conversely, let $X_R \leq R_R$ be a (finitely generated) s.Rickart submodule. By Theorem \ref{gen_triangular_matrix_ring_thm} (ii), $X_R \leq K_R$.  Define $h: X \rightarrow \left( \begin{smallmatrix} 0 & 0 \\ 0 & C \\ \end{smallmatrix} \right) $ by $h( \left( \begin{smallmatrix} 0 & c_1 \\ 0 & c_2 \\ \end{smallmatrix} \right) ) = \left( \begin{smallmatrix} 0 & 0 \\ 0 & c_1 + c_2 \\ \end{smallmatrix} \right)$ where $c_i \in C$. Clearly, $h$ is an $R$-homomorphism. Observe that $h(X)$ is isomorphic to a (finitely generated) $C$-submodule of $C$ and hence $h(X)$ is projective since $C$ is right (semi-)hereditary. Furthermore, if $h$ is injective, we are done. Otherwise, assume $c_1 + c_2 = 0$, i.e., $c_1 = -c_2$. Hence, $ker(h) = \{ \left( \begin{smallmatrix} 0 & -c \\ 0 & c \\ \end{smallmatrix} \right) \in X \ \mid \ c \in C \}$. There exists $g: ker(h) \rightarrow \left( \begin{smallmatrix} 0 & 0 \\ 0 & C \\ \end{smallmatrix} \right)$ defined by $g(\left( \begin{smallmatrix} 0 & -c \\ 0 & c \\ \end{smallmatrix} \right) ) = \left( \begin{smallmatrix} 0 & 0 \\ 0 & c \\ \end{smallmatrix} \right)$ where $g$ is injective. When $C$ is right hereditary, $ker(h)$ is projective. To see that $ker(h)$ is projective when $C$ is right semihereditary, observe that $C$ is right coherent \cite[Example 4.46]{L}. From \cite[pp.~8-9]{G} $ker(h)$ is finitely generated hence projective. Now, the short exact sequence $0 \rightarrow ker(h) \rightarrow X \rightarrow h(X) \rightarrow 0$ splits. Therefore, $X$ is projective.
\end{proof}

\begin{cor} Let $C$ be a commutative domain that is not a field. Then $C$ is a (Pr\"{u}fer) Dedekind domain if and only if every (finitely generated) s.Rickart submodule of $T_2(C)$ is projective.
\end{cor}


\begin{thm} Let $R$ be a ring. Then the following are equivalent.
\begin{enumerate}[$(i)$]
\item $M \in \modr$ is s.Baer if and only if $M$ is projective, and $R$ is right cononsingular.

\item $R_R$ is a nonsingular extending module and all s.Rickart $R$-modules are projective.

\item $Z(R_R) = 0$ and all nonsingular $R$-modules are projective.

\item $R$ is left and right hereditary, left and right Artinian, and the maximal left and right rings of quotients of $R$ coincide.

\item $Z(R_R) = 0$ and all free $R$-modules are extending.

\item $R$ has a ring decomposition $R = \bigoplus_{n=1}^{k} A_i$, where each $A_i$ is Morita equivalent to $T_{n_i}(D_i)$ and each $D_i$ is a division ring.
\end{enumerate}
\end{thm}

\begin{proof}
$(iii) \iff (iv) \iff (v) \iff (vi)$ These equivalences follow from \cite[12.21]{DHSW} and \cite[Theorem 5.23]{G}.

$(i) \Rightarrow (ii)$ Since projective modules are s.Baer, then $R$ is a Baer ring. By \cite[Theorem 2.1]{CK}, $R_R$ is a nonsingular extending module. Since $R$ is a Baer ring, $R$ has the SSIP. From Theorem \ref{ssip}, all s.Rickart modules are s.Baer. Hence all s.Rickart modules are projective.

$(ii) \Rightarrow (iii)$ Clearly, $Z(R_R) = 0$. Let $M \in \modr$ such that $Z(M)=0$. From Theorem \ref{extend}, $M$ is s.Rickart. Therefore $M$ is projective.

$(iii) \Rightarrow (i)$ From the equivalence of $(iii)$ and $(v)$, $R_R$ is nonsingular and extending. From \cite[Theorem 2.1]{CK}, $R_R$ is cononsingular and Baer. Since s.Baer modules are nonsingular, every s.Baer module is projective. By Corollary \ref{baercor}, every projective module is s.Baer.
\end{proof}


\section{Torsion Theory} \label{torsion_section}
From Lemmas \ref{hered} and \ref{sRickart_class_extensions}, the class of s.Baer (s.Rickart) modules, $\sB$ $(\sR)$, is closed under submodules and extensions. Subsequently, under certain conditions (see Theorem \ref{dsums} and Theorem \ref{sbaer_is_torsion_free}), $\sB$ is closed under direct products, thereby making it a torsion-free class \cite[p. 137]{St}. Thus it is natural to ask: \textit{If $\sB$ is a torsion-free class, can we characterize the corresponding s.Baer torsion class?} In this section, we address this question. Notation, terminology, and basic results can be found in \cite{St} and \cite{BKN}.


\begin{thm}
The class of s.Baer (s.Rickart) modules is a torsion-free class if and only if it is closed under direct products. In particular, if $R$ is orthogonally finite or if $R_R$ has the SSIP (e.g., $R_R$ is indecomposable), then the class of s.Baer and s.Rickart modules coincide and form a torsion-free class.
\label{sbaer_is_torsion_free}
\end{thm}

\begin{proof}
Assume $\sB \ (\sR )$ is closed under direct products. From Lemma \ref{hered} and \ref{sRickart_class_extensions} and \cite[p. 140]{St}, $\sB \ ( \sR )$ is torsion-free. The converse follows from \cite[p.~140]{St}. Now suppose that $R$ is orthogonally finite or $R_R$ has the SSIP. By Theorems \ref{ssip} and \ref{orthog_finite}, $\sB = \sR$. From Theorem \ref{dsums}, $\sB$ is closed under direct products.
\end{proof}

Looking back to Corollary \ref{baercor}, we can gain a bit of perspective on $\sR$. In general, this class is not closed under arbitrary products thus preventing it from being a torsion-free class (see \cite[Theorem 2.9]{LC}). Observe that if $\sR$ is closed under arbitrary direct products, then $\sR$ would be a torsion-free class and many of the results in this section would hold true for it. Alternatively, we could try and view $\sR$ as a torsion class since it is closed under direct sums. Unfortunately, undifferentiated from the s.Baer modules, the s.Rickart modules are not closed under homomorphic images. Recalling Example \ref{counterexample}, we see that the class of s.Baer modules is strictly smaller than the class of s.Rickart modules. 

Setting aside $\sR$, let us concentrate further on $\sB$ and introduce some terminology and notation. We will denote the torsion theory associated with the class of s.Baer modules (when it exists) as $(\mathcal{T}_{\beta}, \mathcal{F}_\beta)$ where $\storfree  = \sB$ and $\beta$ is its associated idempotent radical. \textit{Henceforth when we speak of the s.Baer torsion theory $(\stor, \storfree)$, we are implicitly assuming the class of s.Baer modules is closed under direct products} (see Theorem \ref{sbaer_is_torsion_free}). Also recall from \cite[p.~141]{St} that a \textit{hereditary} torsion-theory $(\mathcal{T}, \mathcal{F})$ is one in which $\mathcal{T}$ is closed under submodules (equivalently, $\mathcal{F}$ is closed under injective hulls); and $(\tor, \torf)$ is \textit{stable} if $\tor$ is closed under injective hulls. Example \ref{last_ex}(iii) shows that, in general, $\stor$ is not hereditary. 

A \textit{(left) right duo} ring $R$ is a ring in which every (left) right ideal of $R$ is a two sided ideal. With a routine argument, one can show that in a right duo ring all idempotents are central. Recall that, in general, the s.Baer property does not pass to essential extensions (Example \ref{last_ex} (iii)).


\begin{lem} \label{duo_hered_torsion}
Let $S_R \ess M_R$ such that $S$ is a s.Baer module. If any of the following conditions are satisfied, then $M$ is s.Rickart.
\begin{enumerate}[(i)]
\item $R_R$ is $\mathcal{G}$-extending.
\item $\mathcal{B} (R) = \semicentral$ (e.g., $R$ is semiprime) and $r_R(m) \ideal R$ for each $m \in M$.
\item $R$ is a right duo ring.
\end{enumerate}
\end{lem}

\begin{proof}
$(i)$ Since $Z(S) = 0$, $Z(M) = 0$. From Theorem \ref{extend} (i), $M$ is s.Rickart.

$(ii)$ Let $m \in M - S$ and  let $L=m^{-1}S= \{ r \in R \ | \ mr \in S \}$. Since $mL_R \leq M_R$ and $mL \subseteq S$, $r_R(mL) = cR$ where $c \in \semicentral = \mathcal{B} (R)$. Then $0 = mLc = mcL$. Since $L_R \ess R_R$ and $Z(M)=0$, $mc = 0$. So $cR \subseteq r_R(m) = r_R(mR) \subseteq r_R(mL) = cR$. Therefore $r_R (m) = cR$, hence $M$ is s.Rickart.

$(iii)$ If $R$ is right duo then all idempotents are central and $r_R(m) \ideal R$ for all $m \in M$. So this part follows from $(ii)$.
\end{proof}


\begin{prop} \label{commutative_hered_torsion}
Suppose $R_R$ has the SSIP or is orthogonally finite. Then if $R$ is right duo or $R_R$ is $\mathcal{G}$-extending (e.g., $R$ is a commutative Noetherian ring or $R_R$ is nonsingular extending), $(\mathcal{T}_\beta, \mathcal{F}_\beta)$ is a hereditary torsion theory.
\end{prop}
\begin{proof}
Let $M \in \sB$ and consider $E(M)$. By Lemma \ref{duo_hered_torsion}, $E(M) \in \sR$. Theorem \ref{ssip} or Theorem \ref{orthog_finite} imply $E(M) \in \sB$. Thus, $\mathcal{F}_\beta$ is closed under injective hulls, so $(\mathcal{T}_\beta, \mathcal{F}_\beta)$ is a hereditary torsion theory.
\end{proof}


\begin{prop} \label{torsion_class_extensions}
The s.Baer torsion class $\mathcal{T}_\beta$ is closed under essential extensions. Hence, $(\stor, \storfree)$ is a stable torsion theory.
\end{prop}
\begin{proof}
Let $T \in \stor$ and $E(T) = M$. By Proposition \ref{essential_hom}, $Hom_R(M,F) = 0$ for any $F \in \storfree$. So $M \in \stor$. Thus, $(\stor, \storfree)$ is stable.
\end{proof}

Recall from \cite{BMR}, a module $M$ is \textit{FI-extending} if every fully-invariant submodule is essential in a direct summand of $M$.

\begin{cor} \label{FIsplitting} Let $( \mathcal{T} , \mathcal{F} )$ be a torsion theory with a radical $\rho$ for which $\mathcal{T}$ is closed under essential extensions. If $M \in \modr$ is FI-extending, then $\rho (M)$ splits-off. In particular, $\B (M)$ splits-off in any FI-extending module.
\end{cor}

\begin{proof}
For any $M \in \modr$, $\rho (M)$ is a fully-invariant submodule of $M$. Since $M$ is FI-extending and $\mathcal{T}$ is closed under essential extensions, $\rho (M) = D$ where $ D \dsum M$. Lastly, observe $( \stor, \storfree )$ is a torsion theory which satisfies the hypotheses (see Proposition \ref{essential_hom}).
\end{proof}

The Goldie torsion theory consists of the stable hereditary torsion class $\mathfrak{G} = \{ N \in \mathcal{M}od_R \ | \ Z(N)\leq^{ess} N \}$, the hereditary torsion-free class $\mathfrak{F} = \{ N \in \mathcal{M}od_R \ | \ Z(N)=0 \}$, and its associated left exact radical $Z_2$ \cite[pp. 148,158]{St}. Since every s.Baer module is nonsingular, $\mathcal{F}_\beta$ is contained in $\mathfrak{F}$ and $\mathfrak{G}$ is contained in $\mathcal{T}_\beta$. For an alternate proof of Proposition \ref{torsion_class_extensions} and more on torsion theories associated with the Goldie Torsion Theory, see \cite{T}.

Up to this point, most of our previous results concern $\storfree$. The remainder of this paper focuses on $\stor$. Here we give conditions for when a module is a s.Baer torsion module.

\begin{prop} \label{cyclic_torsion}
Let $mR$ be a nonzero cyclic $R$-module where $r_R(m)\nsubseteq eR$ for any nontrivial $e=e^2 \in R$. Then $mR \in \stor$.
\end{prop}

\begin{proof}
Suppose, to the contrary, $Hom(mR,F) \neq 0$ for $F \in \storfree$ and let $h:mR \rightarrow F$ be a nonzero $R$-homomorphism. We know $mR \iso R/r_R(m)$ and we denote $r_R(m) = H$. Now let $ker(h) \iso K/H \leq R/H$. By the second isomorphism theorem, $R/K \iso \frac{R/H}{K/H} \iso h(mR) \in \storfree$. Then $r_R(1+K) = K = eR$ for some $e=e^2 \in R$. By correspondence, $H \subseteq K$, i.e., $r_R(m) \subseteq eR$ contrary to the hypothesis.
\end{proof}


Note that in Example \ref{last_ex}(iii), $\beta (R_R) = \left( \begin{smallmatrix} 1 & 0 \\ 0 & 0 \\ \end{smallmatrix} \right) R$ but $r_R( \left( \begin{smallmatrix} 1 & 0 \\ 0 & 0 \\ \end{smallmatrix} \right) ) = \left( \begin{smallmatrix} 0 & 0 \\ 0 & 1 \\ \end{smallmatrix} \right) R$. Thus the converse of Proposition \ref{cyclic_torsion} is false. 


\begin{cor} 
Let $M \in \modr$ and $R_R$ be indecomposable.
\begin{enumerate}[(i)]
\item $\sum_{r_R(m) \neq 0} mR \leq \B (M)$, and if $\sum_{r_R(m) \neq 0} mR = 0$ then $\B (M) = 0$.
\item If $(\stor, \storfree)$ is hereditary (e.g., $R$ is commutative) then $\sum_{r_R(m) \neq 0} mR \ess \B (M)$.
\end{enumerate}
\end{cor}

\begin{proof}
$(i)$ This is immediate from Propositions \ref{cyclic_torsion} and \ref{indecomposable}.

$(ii)$ Assume there is a nonzero $y \in R$ such that $yR \cap \sum_{r_R(m) \neq 0} mR = 0$. Then for all $0 \neq r \in R$, $yr \neq 0$ (otherwise $y \in \sum_{r_R(m) \neq 0} mR$) and $r_R(yr) = 0$. Contrary to our hypothesis, $yR$ is s.Baer by Proposition \ref{indecomposable}. Thus, $\sum_{r_R(m) \neq 0} mR \ess \B (M)$.
\end{proof}


\begin{lem} \label{prerad}
\cite[p. 16]{BKN} Let $\rho$ be a (pre-)radical. Then $M \rho (R) \subseteq \rho (M)$ for any $M \in \mathcal{M}od_R$. In particular, if $M$ is projective, then $M \rho (R) = \rho (M)$.
\end{lem}

\begin{defn} Let $M \in \modr$. The s.Baer (s.Rickart) core of M, $\sbcore (M)$ $( \srcore (M) )$, is the nonempty subset $\{ s \in M \mid sR \in \sB \ ( \sR )  \}$.
\end{defn}

Note $0 \in \sbcore (M)$ ($\srcore (M)$) and, by Lemma \ref{hered}, $\sbcore (M)$ ($\srcore (M)$) is closed under scalar multiplication. If $e \in \semicentral$ and $M \in \modr$, then $1-e \in \rsemicentral$, $eR(1-e) \ideal R$, and $M(1-e) \leq M$ (see \cite[p. 5]{BPR}).

In the following results, we investigate properties of modules with a nonzero s.Baer radical and, in several instances, we compute the s.Baer radical of a module.


\begin{thm} \label{R_nonzero_radical}
Let $M \in \modr$ and $\B (M) \neq M$. Then:
\begin{enumerate}[(i)]
\item There exists $c \in \semicentral$ such that $r_R( \sbcore (M) ) \cap r_R( \sbcore (R) ) = cR \neq R$, $\B (R) \leq cR$, and $(1-c)R \in \storfree$. If $\B (M) = M \B (R)$ (e.g., $M$ is projective), $\B (M) \leq McR$.

\item If $\B (R) \ess cR$, then $\B (R) = cR$.

\item Suppose $\B (R) \ness cR$ and $K$ is any relative complement of $\B (R)$ in $cR$. Then $K \leq cR(1-c)$ and there exists a projective submodule $P$ of $K$ such that $\B (R) \oplus P \ess cR$. 
\end{enumerate}
\end{thm}

\begin{proof}
$(i)$ Since $\B (M) \neq M$, then $\B (R) \neq R$. For if $\B (R) = R$, then $M = M \B (R) \leq \B (M) \leq M$, by Lemma \ref{prerad}. From Proposition \ref{torsion_class_extensions}, $\B (M) \ness M$ and $\B (R) \ness R$. Hence, $\sbcore (M)$ and $\sbcore (R)$ are nonzero. Let $I$ be an indexing set for $\sbcore (M)$. Since we are assuming $\storfree$ exists, $\prod_I s_i R \in \sB$ where $s_i \in \sbcore (M)$. Define $g \in \prod_I s_iR$ by $f(i) = s_i$. Then $r_R ( \sbcore (M) ) = r_R(g) = eR \neq R$ where $e \in \semicentral$ since $\sbcore (M)$ is closed under scalar multiplication. By Lemma \ref{prerad}, $sR \B (R) \subseteq \B (sR) = 0$. Hence, $\B (R) \subseteq r_R (sR) \subseteq r_R (s)$ for all $s \in \sbcore (M)$. Then $\B (R) \subseteq \bigcap_{s \in \sbcore (M)} r_R (s) = r_R (\sbcore (M) ) = eR$.

A similar proof using $\sbcore (R)$ gives us $\B (R) \subseteq \bigcap_{s \in \sbcore (R)} r_R (s) = r_R( \sbcore (R) ) = fR$ for some $f \in \semicentral$. Take $c = fe \in \semicentral$. Then $\B (R) \subseteq cR$ and $(1-c)R \in \storfree$. By Lemma \ref{prerad}, $\B (M) \leq McR$.

$(ii)$ This part follows from Proposition \ref{torsion_class_extensions}.

$(iii)$ Clearly $K = cK$. Observe that $K \subseteq \sbcore (R)$. Hence, $0 = Kfe = Kc$. So $K = K(1-c)$. Thus, $K \leq cR(1-c)$. Theorem \ref{essprojective} ensures the existence of the desired $P$.
\end{proof}

\begin{cor} Let $M \in \modr$ such that $\B (M) = M \B (R)$ (e.g., $M$ is projective) and $\B (M) \neq M$. Taking $c$ as in Theorem \ref{R_nonzero_radical}, we have:
\begin{enumerate}[(i)]
\item $\B ( M(1-c) ) \leq McR(1-c) \leq J(M)$ where $J(M)$ is the Jacobson radical of $M$. Hence, if $J(M) = 0$ then $M(1-c) \in \storfree$.

\item If every semicentral idempotent is central (e.g., $R$ is semiprime), then $M = Mc \oplus M(1-c)$ where $Mc = \B (M)$ and $M(1-c) \in \storfree$.

\item If $M$ is projective and $R$ is semiprime then $\B (M)$ splits-off.


\end{enumerate}
\end{cor}

\begin{proof}
$(i)$ By Lemma \ref{prerad}, $\B (M(1-c)) \leq \B (M) \cap M(1-c) \leq McR \cap M(1-c) = McR(1-c) \leq M J(R) \leq J(M)$.  

$(ii)$ From Theorem \ref{R_nonzero_radical} (ii) and (iii), $\B (R) = cR$. Then $\B (M) = McR = Mc$ so $M(1-c) \in \storfree$.

$(iii)$ This part follows from $(ii)$.
\end{proof}

\begin{cor} \label{semicentral_reduced}If $R$ is semicentral reduced (e.g., $R$ is a prime ring) and $M$ is projective then either $\B (M) = M$ or $\B (M) = 0$.
\end{cor}

\begin{proof}
This result is an immediate consequence of Theorem \ref{R_nonzero_radical} (i).
\end{proof}

Even in the Goldie torsion theory, we do not see such a dichotomy as in Corollary \ref{semicentral_reduced} for projective modules. There exists $R_R$ indecomposable (hence semicentral reduced) where $0 \neq Z_2 (M) \subsetneq M$ (e.g., $M=R$ is the trivial extension of $\Z_4$ by $\Z$).


\begin{prop} Let $R = \left( \begin{smallmatrix} A & M \\ 0 & C \\ \end{smallmatrix} \right)$ where $C$ is semicentral reduced. Assume that for each $1 \neq a \in \mathcal{S}_l (A)$ there exists $m \in M$ such that $am \neq m$ (e.g., $_AM$ is faithful). Then:
\begin{enumerate}[(i)]

\item $\B (R) \neq R$ if and only if there exists $1 \neq e \in \mathcal{S}_l (A)$ such that $\B (R) \leq \left( \begin{smallmatrix} eA & eM \\ 0 & 0 \\ \end{smallmatrix} \right)$ if and only if $C$ is a Baer ring.

\item $\B (R) \subseteq \left( \begin{smallmatrix} eA & 0 \\ 0 & 0 \\ \end{smallmatrix} \right)$ where $1 \neq e \in \mathcal{S}_l (A)$ if and only if $\B (R) \cap \left( \begin{smallmatrix} 0 & M \\ 0 & C \\ \end{smallmatrix} \right) = \{0 \}$.

\item Assume for each $0 \neq ea \in eA$ there exists $k \in M$ such that $0 \neq eak$ (e.g., $_AM$ is faithful) where $1 \neq e \in \mathcal{S}_l (A)$. 
Let $\left( \begin{smallmatrix} 0 & X \\ 0 & 0 \\ \end{smallmatrix} \right) = 
\B (R) \cap \left( \begin{smallmatrix} 0 & M \\ 0 & C \\ \end{smallmatrix} \right)$. 
If $\B (R) = \left( \begin{smallmatrix} eA & eM \\ 0 & 0 \\ \end{smallmatrix} \right)$, then $X \neq 0$. If $X_C \ess eM_C$, then 
$\B (R) = \left( \begin{smallmatrix} eA & eM \\ 0 & 0 \\ \end{smallmatrix} \right)$.
\end{enumerate}
\label{compute_radical}
\end{prop}

\begin{proof}
$(i)$ Assume $\B (R) \neq R$. From Theorem \ref{R_nonzero_radical} (i), there exists $c \in \semicentral$ with $c \neq 1$ such that $\B (R) \leq cR$. By \cite[Proposition 1.3]{ABT1}, there exists $e \in \mathcal{S}_l (A)$ and $f \in \mathcal{S}_l (C)$ such that $c = \left( \begin{smallmatrix} e & 0 \\ 0 & f \\ \end{smallmatrix} \right)$ and $emf = mf$ for all $m \in M$. Observe that $f \in \{ 0, 1 \} \subseteq C$. Assume that $f = 1$. Then $em = m$ for all $m \in M$. By hypothesis, $e = 1$, a contradiction. Hence $f = 0$, so $\B (R) \leq \left( \begin{smallmatrix} eA & eM \\ 0 & 0 \\ \end{smallmatrix} \right)$. This in turn implies that $\left( \begin{smallmatrix} 0 & 0 \\ 0 & C \\ \end{smallmatrix} \right) \in \storfree$, i.e., $C$ is a Baer ring. The last equivalence yields that $\B (R) \neq R$ since there exists a nonzero homomorphism $R \rightarrow R / \left( \begin{smallmatrix} A & M \\ 0 & 0 \\ \end{smallmatrix} \right) \iso \left( \begin{smallmatrix} 0 & 0 \\ 0 & C \\ \end{smallmatrix} \right)$.

$(ii)$ The proof is routine.

$(iii)$ Let $0 \neq t = \left( \begin{smallmatrix} ea & em \\ 0 & 0 \\ \end{smallmatrix} \right) \in 
\left( \begin{smallmatrix} eA & eM \\ 0 & 0 \\ \end{smallmatrix} \right)$. Assume that $\B (R) = 
\left( \begin{smallmatrix} eA & eM \\ 0 & 0 \\ \end{smallmatrix} \right)$. If $ea = 0$ then $ 0 \neq em \in X$. If $ea \neq 0$ there exists $k \in M$ such that $0 \neq eak \in X$. Now assume that $X_C \ess eM_C$. 
If $0 \neq em$ there exists $\gamma \in C$ such that $0 \neq em\gamma \in X$. Hence, $0 \neq t 
\left( \begin{smallmatrix} 0 & 0 \\ 0 & \gamma \\ \end{smallmatrix} \right) \in \B (R)$. If $0 = em$, there exists $n \in M$ such that $0 \neq ean \in eM$. Again there exists $\delta \in C$ such that $0 \neq ean\delta \in X$. So $0 \neq t \left( \begin{smallmatrix} 0 & n\delta \\ 0 & 0 \\ \end{smallmatrix} \right) \in \B (R)$. Consequently, $\B (R)_R \ess \left( \begin{smallmatrix} eA & eM \\ 0 & 0 \\ \end{smallmatrix} \right)_R$. By Proposition \ref{torsion_class_extensions}, $\B (R) = \left( \begin{smallmatrix} eA & eM \\ 0 & 0 \\ \end{smallmatrix} \right)$.
\end{proof}

In Proposition \ref{compute_radical}, if $\B (R) \neq 0$ and $_AM$ is faithful, then $\B (R) \nsubseteq \left( \begin{smallmatrix} eA & 0 \\ 0 & 0 \\ \end{smallmatrix} \right)$.

\begin{ex} \label{last_ex} The following are generalized triangular matrix rings, $R = \left( \begin{smallmatrix} A & M \\ 0 & C \\ \end{smallmatrix} \right)$ where $C$ is semicentral reduced, $_AM$ is faithful, $M_C$ is uniform, and $R$ is orthogonally finite or $R_R$ has the SSIP. Hence, $R$ satisfies all the hypotheses of Proposition \ref{compute_radical}.
\end{ex}
\begin{enumerate}[(i)]
\item $R = T_2 (A)$ where $A$ is a right uniform local ring with a nonzero nilpotent right ideal (e.g., $\Z_{p^n}$ for $n>1$). By the comment after Corollary \ref{FIsplitting}, $R = Z_2 (R) = \B (R)$. $R$ is orthogonally finite but, in general, $R_R$ does not have the SSIP.

\item Let $M = C$ be a division ring and $A$ a subring of $C$. By Corollary \ref{gen_triangular_matrix_cor} (i),  $\B (R) = 0$. $R$ is orthogonally finite and $R_R$ has the SSIP (see Lemma \ref{triangular_ssip_lemma}).

\item Let $C$ be a right Ore domain that is not a division ring, $M$ its classical right ring of quotients, and $A$ a subring of $C$. Then $\B (R) = \left( \begin{smallmatrix} A & M \\ 0 & 0 \\ \end{smallmatrix} \right)$. $R$ is orthogonally finite, and $R_R$ has the SSIP (see Lemma \ref{triangular_ssip_lemma}). $\stor$ is not hereditary since $M_C$ is s.Baer. Therefore, $\sB$ is not closed under essential extensions.
\end{enumerate}

\section*{Acknowledgements}
The fact that $R$ is not SSIP in Example \ref{counterexample} is due to a referee.

\section*{Preprint}
Preprint of an article published in [J. Algebra Appl. Volume 14, Issue 8, 2015] [DOI: 10.1142/S0219498815501315]  [\textcopyright World Scientific Publishing Company] [\url{http://www.worldscientific.com/worldscinet/jaa}]


\bibliographystyle{alpha}
\bibliography{jobname}

\end{document}